\documentclass[12pt,reqno]{amsart}
\usepackage{amssymb,amsmath,amsthm,newlfont,enumerate}
\newtheorem{thm}{Theorem}[section]

\newtheorem{cor}[thm]{Corollary}
\newtheorem{lem}[thm]{Lemma}
\newtheorem{rem}[thm]{Remarks}
\newcommand{\T}{\mathbb{T}}

\newcommand{\D}{\mathbb{D}}

\newcommand{\cD}{{\mathcal{D}}}

\newcommand{\cA}{{\mathcal{A}}}

\renewcommand{\Re}{\mathrm{Re }}
\renewcommand{\Im}{\mathrm{Im }}

\newcommand{\hH}{\mathrm{H^{2}   }}

\newcommand{\caap}{\mathrm{cap}}

\newcommand{\Hol}{\mathrm{Hol}}

\newcommand\ZZ{\mathbb{Z}}
\newcommand\TT{\mathbb{T}}
\newcommand\DD{\mathbb{D}}

\begin{document}

\title[Level sets and composition operators]{Level sets and Composition operators on the Dirichlet space}

\author[El-Fallah]{O. El-Fallah$^1$}
\address{$^1$ D\'{e}partement de Math\'ematiques\\ Universit\'e Mohamed V\\  B.P. 1014 Rabat\\Morocco}
\email{elfallah@fsr.ac.ma}

\author[Kellay]{K. Kellay$^2$}
\author[Shabankhah]{M. Shabankhah$^2$}
\author[Youssfi]{H. Youssfi$^2$ }
\address{$^2$ CMI\\ LATP\\ Universit\'e de Provence\\ 39, rue F. Joliot-Curie\\   13453  Marseille \\France}
\email{kellay@cmi.univ-mrs.fr}
\email{mshabank@cmi.univ-mrs.fr}
\email{youssfi@cmi.univ-mrs.fr}
\thanks{{1. Research partially supported by  a grant from Egide Volubilis}}
\thanks{{2. Research partially supported by grants from Egide Volubilis  and ANR Dynop}}

\maketitle

\begin{abstract} We consider composition operators in the Dirichlet space of the unit disc in the plane. Various criteria on boundedness, compactness and Hilbert-Schmidt class membership are established.  Some of these criteria are shown to be optimal. 
\end{abstract}
\section{Introduction}

In this note we consider composition operators   in the Dirichlet space of  the unit disc.
A comprehensive study of composition operators in function spaces and their spectral behavior could be
found in \cite{CM,S,Z}. See also \cite{GG, GG2, GG1, T, WX, Zo}
for a treatment of some of the questions addressed in this paper.

Let  $\DD$ be the unit disc in the complex plane and let  $\TT=\partial \DD$ be its boundary.
We denote by $\cD$ the classical Dirichlet  space. This is the space of all analytic functions
$f$ on  $\DD$ such that
 $$\cD(f):=\int_\DD|f'(z)|^2 \,dA(z)<\infty,$$
where $dA(z)=dxdy/\pi$ stands for the normalized area measure in $\DD$. We call $\cD(f)$ the Dirichlet integral of $f$.
The space $\cD$ is endowed with the norm
$$\|f\|_{\cD}^2:=|f(0)|^2+\cD(f).$$

It is standard that a function $f(z)=\sum_{n=0}^\infty \widehat{f}(n) \, z^n$,
holomorphic on $\DD$, belongs to $\cD$
if and only if
$$
\sum_{n\geq 0}|\widehat{f}(n)|^2\, (1+n) < \infty,
$$
and that this series defines an equivalent norm on $\cD$.

Since the Dirichlet space is contained in the Hardy space $\hH (\DD)$,
every function $f \in \cD$ has non-tangential limits $f^*$ almost everywhere on $\TT$.
In this case, however, more can be said.
Indeed, 
Beurling \cite{C} showed that 
if $f \in \cD$
then $f^*(\zeta) = \lim_{r\to 1} f(r\zeta)$ exists for $\zeta \in \TT$ outside  of a set of logarithmic capacity zero.

Let $\varphi$ be a holomorphic self-map of $\D$.  The composition operator $C_\varphi$
on $\cD$ is defined by
$$C_\varphi(f)=f\circ\varphi,\qquad f\in \cD.$$
We are interested herein in describing the spectral properties of the
composition operator $C_\varphi$, such as compactness and Hilbert-Schmidt class membership,
in terms of the size of the level set of  $\varphi$.
For $s\in (0,1)$, the level set $E_\varphi(s)$ of $\varphi$ is given by
$$E_\varphi(s)=\{\zeta\in \TT \text{: } |\varphi(\zeta)|\geq s\}.$$
We give new characterizations of Hilbert-Schmidt class membership in the case of
the Dirichlet space. We  also establish the sharpness of these results.

\section{A general criterion}
For $\alpha>-1$, $dA_\alpha$ will denote the finite measure on $\DD$
given by
$$
dA_\alpha(z) := (1+\alpha)(1-|z|^2)^\alpha dA(z). 
$$
For $p\geq 1$ and $\alpha>-1$, the weighted  Bergman space
$\cA^p_\alpha$ consists of the holomorphic functions $f$ on $\DD$ for which
$$\|f\|_{p,\alpha} := \Big[\int_\DD |f(z)|^p \, dA_\alpha(z)\Big]^{1/p} \,<\, \infty.$$

We denote by  $\cD_\alpha^p$ the space consisting of analytic functions $f$  on $\DD$ such that
$$\|f\|_{\cD_\alpha^p}^{p} := |f(0)|^p+\|f'\|_{p,\alpha}^{p} \,<\, \infty.$$
Appropriate choices of the parameter  $\alpha$ give, with equivalent norm,  all the standard holomorphic function spaces. Indeed, 
The Hardy space $\hH$ can be identified  with $\cD_1^2$.  
The classical Besov space is precisely $\cD^{p}_{p-2}$, and if $p<\alpha+1$,
$\cD_\alpha^p=\cA^{p}_{\alpha-2}$.  Finally, the classical Dirichlet space $\cD$ is identical to $\cD^{2}_{0}$.

We recall that, by the reproducing formula, one has
\begin{equation}\label{kernel}
f(z)=\int_\DD \frac{f(w)}{(1-\overline{w}z)^{2+\alpha}} \, dA_\alpha(w),\qquad z\in \DD,
\end{equation}
for every $f\in \cA^p_\alpha$ (see \cite{Z}).
\begin{lem}\label{besovestimate} Let $p\geq 1$ and let $\sigma>-1$.
Then, there exists a constant $C$ depending only on $p$ and $\sigma$  such that
\begin{eqnarray*}
|f(z)|^p &\leq & C\int_\DD \frac{|f(\lambda)|^p}{|1-\overline{\lambda}z|^{2+\sigma}} \, dA_{\sigma}(\lambda),
\end{eqnarray*}
for every $f\in \cA_\sigma^p$ and $z\in \DD$.
\end{lem}
\begin{proof}
By the above reproducing formula,
\begin{eqnarray*}
\frac{f(z)}{1-z\overline{w}} & = & \int_\DD \frac{f(\lambda)}{1-\lambda\overline{w}}\,
\frac{dA_{\sigma }(\lambda)}{(1-\overline{\lambda}z)^{2+\sigma }},\qquad z,w\in \DD,
\end{eqnarray*}
for every $f\in \cA_\sigma^p$.
By  H\"{o}lder's inequality, with $q=p/(p-1)$, 
$$
\frac{|f(z)|^p}{|1-z\overline{w}|^{p}} \leq  \int_\DD \frac{|f(\lambda)|^p  dA_{\sigma }(\lambda)}{|1-\overline{\lambda}z|^{2+\sigma }} \times
 \left(\int_\DD \frac{dA_{\sigma}(\lambda)}{|1-\lambda\overline{w}|^{ q}  |1-\lambda\overline{z}|^{(2+\sigma) p}}\right)^{\frac{p}{q}}.$$
Taking  $w=z$, and using the standard estimate (\cite[Lemma 3.10]{Z})
 \begin{equation}\label{integral}
 \int_\DD \frac{dA_c(\lambda)}{|1-z\overline{\lambda}|^{2+c+d}} \asymp \frac{1}{(1-|z|^2)^d}, \qquad \text{if  }d>0 \text{ , }Êc>-1,
 \end{equation}
we get the desired conclusion.
  \end{proof}
For $\lambda\in \DD$, consider the test function
 $$F_{\lambda,\beta}(z)=\frac{1}{(1-\overline{\lambda}z)^{1+\beta}},\qquad z\in \DD.$$
If  $\beta \geq 0$ is chosen such that $\delta:=\delta(p,\alpha,\beta)=2+\beta-(2+\alpha)/p>0$, 
 by \eqref{integral}, we have
 $$\|F_{\lambda,\beta}\|_{\cD_\alpha^p}^{p}\asymp (1-|\lambda|^2)^{-p\delta}.$$

The following theorem unifies and generalizes the previously known results of
MacCluer \cite[Theorem 3.12]{CM},  Tjani \cite[Theorem 3.5]{T} and Wriths-Xiao \cite[Theorem 3.2]{WX}
on Hardy, Besov and weighted Dirichlet spaces, respectively.

As mentioned before, the proof we provide here is short and simple.
\begin{thm}\label{cg} Let $p>1$. Suppose $\varphi\in \cD^p_\alpha$ satisfies $\varphi(\DD)\subset\DD$.
Fix $\beta\geq 0$ such that $\delta:=\delta(p,\alpha,\beta)=2+\beta-(2+\alpha)/p> 0$. Then \\
\begin{enumerate}
\item[(a)] $C_\varphi$ is bounded on $\cD_{\alpha}^{p} \displaystyle \iff \sup_{\lambda\in\DD}\;(1-|\lambda|^2)^{\delta}\|F_{\lambda,\beta}\circ\varphi\|_{\cD_\alpha^p}<\infty$;\\
\item[(b)] $C_\varphi$ is  compact on  $\cD_\alpha^p \displaystyle \iff \lim_{|\lambda|\to 1} (1-|\lambda|^2)^{\delta}\|F_{\lambda,\beta}\circ\varphi\|_{\cD^p_\alpha}=0.$
\end{enumerate}
\end{thm}
\begin{proof}   To prove
(a), we observe that if  $C_\varphi$ is bounded , then
$$
\displaystyle \|F_{\lambda,\beta}\circ\varphi\|_{\cD_\alpha^p}=O((1-|\lambda|^2)^{-\delta}).
$$
For the converse  we may assume, without loss of generality, that $\varphi$ fixes the origin.
It follows from  Lemma 2.1 that, for $f\in \cD_\alpha^p$,
\begin{eqnarray*}
& &  \int_\DD|\varphi'(z)|^p|f'(\varphi(z))|^p \, dA_\alpha(z) \\
&\leq &C\int_\DD|\varphi'(z)|^p\Big(\int_\DD \frac{|f'(\lambda)|^p}{|1-\overline{\lambda}\varphi(z)|^{(2+\beta)p}} \, dA_{2p+\beta p -2}(\lambda)\Big)\, dA_\alpha(z)\\
&=&       C\int_\DD |f'(\lambda)|^p \Big((1-|\lambda|^2)^{2p+\beta p-2-\alpha}\int_\DD \frac{ |\varphi'(z)|^p}{|1-\overline{\lambda}\varphi(z)|^{(2+\beta)p}} \, dA_\alpha(z)\Big) \, dA_\alpha(\lambda)    \\
&=&  C \int_\DD|f'(\lambda)|^p (1-|\lambda|^2)^{p\delta} \| (F_{\lambda,\beta} \circ \varphi)' \|_{p,\alpha}^{p} \, dA_\alpha(\lambda).
\end{eqnarray*}
Therefore  part  (a) follows.

(b) Without loss of generality we assume that $\varphi(0)=0$.
Note that $C_\varphi$ is compact on $\cD_\alpha^p$ if and only if for every bounded
sequence $(f_n)_n \subset \cD_\alpha^p$ such that $f_n \to 0$  uniformly on compact subsets of $\DD$,
we have $\|C_\varphi(f_n)\|_{\cD_\alpha^p}\to 0$, as $n \to \infty$.

 Suppose that $C_\varphi$ is compact.  Since $(1-|\lambda|^2)^{\delta}F_{\lambda,\beta}\to 0$ uniformly on compact subsets of the unit disc, as $|\lambda| \to 1$,
 we see that 
 $$\|C_\varphi(F_{\lambda,\beta})\|_{\cD_\alpha^p}=o((1-|\lambda|^2)^{-\delta}).$$

Conversely, assume that $\displaystyle \lim_{|\lambda| \to 1}(1-|\lambda|^2)^{\delta} \|F_{\lambda,\beta} \circ \varphi \|_{\cD_\alpha^p}=0$.
Let $(f_n)_n $ be a bounded sequence of $\cD_\alpha^p$ such that $f_n \to 0$  uniformly on compact sets. Since $f'_n \to 0$  uniformly on compact sets,  it follows from  the proof of  part $(a)$  and the hypothesis that, for $r$ close enough to 1,
\begin{multline*}
  \|C_\varphi (f_n) \|_{\cD_\alpha^p}^p - |f_n(0)|^p
  \leq   \int_{r\D}  |f'_n(\lambda)|^p  (1-|\lambda|^2)^{p\delta}\|( F_{\lambda,\beta} \circ \varphi) \|_{p,\alpha}^{p}dA_\alpha(\lambda)
\\ 
 +   \int_{\D \setminus r\D } |f'_n(\lambda)|^p (1-|\lambda|^2)^{p\delta} \| (F_{\lambda,\beta} \circ \varphi)' \|_{p,\alpha}^{p}dA_\alpha(\lambda) \to 0, \quad n\to \infty
\end{multline*}
which finishes the proof.
\end{proof}

The following result is an immediate consequence of Theorem \ref{cg}.

\begin{cor}\label{limitdirichlet} Let $\varphi: \DD \to \DD$ such that $\varphi\in \cD$. \\
\begin{enumerate}
\item[(a)] If $\displaystyle \sup_{n\geq 1} \, \cD(\varphi^n)<\infty$, then $C_\varphi$ is bounded;\\
\item [(b)] If  $\displaystyle \lim_{n\to \infty} \, \cD(\varphi^n)=0$, then $C_\varphi$ is compact.

\end{enumerate}
\end{cor}
\begin{proof}
We consider the test function $F_{\lambda,0}$ with $\beta=\alpha=0$ and $p=2$. Both (a) and (b) follow from the following inequality:
\begin{eqnarray*}
\cD(F_{\lambda,0}\circ\varphi)&\leq& 2\, (\, 1-|\lambda|^2)^2 \, \int_\DD\frac{|\varphi'(z)|^2}{(1-|\lambda|^2\varphi(z)|^2)^4}\, dA(z)\\
&\leq& c\, (1-|\lambda|^2)^2 \, \sum_{n\geq 0} \,(n+1)^3 \, |\lambda|^{2n}\, \int_\DD |\varphi'(z)|^2| \, \varphi^n(z)|^2\, dA(z)\\
&=& c\, (1-|\lambda|^2)^2 \, \sum_{n\geq 0} \, (1+n)\, |\lambda|^{2n} \, \cD(\varphi^{n+1})\\
&\leq & c\, \limsup_{n\to\infty}\cD(\varphi^{n+1}).
\end{eqnarray*}
\end{proof}

\begin{rem}
\end{rem}
1. The compactness criterion for $C_\varphi$ in the  Bloch space is equivalent
to $\|\varphi^n\|_{\mathcal{B}}\to 0$ as was shown in \cite{WZZ} (see also \cite{M,T}).
In the case of the Hardy space $\hH $, however,
we know that if $C_\varphi$ is compact on $\hH $ then $\|\varphi^n\|_{\hH }\to 0$
but the converse does not hold \cite{CM}. Note that as before in the proof of Corollary \ref{limitdirichlet} ($\beta=0, \alpha=1$ and $p=2$) if $\|\varphi^n\|_{\hH }=o(1/\sqrt{n})$, then $C_\varphi$ is compact on $\hH $.

\vspace{1em}

2. The characterization of compact composition operators on the Dirichlet space
in terms of Carleson measures can be found in \cite{CM, T, Zo}. A positive Borel measure $\mu$ given on $\DD$ satisfying
$$\int_\DD |f(z)|^2\, d\mu(z)\leq \|f\|_{2,0}^2, \qquad f\in \cA_0^2,$$
is called a Carleson measure for $\cA_0^2$, i.e., the identity map $i_0:\cA_0^2\to \textrm{L}^2(\mu)$ is a  bounded operator.
 Such a measure has the following equivalent properties (see \cite[Theorem 7.4]{Z}).
A positive  Borel measure $\mu$ is a Carleson measure for $\cA_0^2$ if and only if
$$\sup_{\lambda\in \DD}(1-|\lambda|^2)^{2}\int_\DD\frac{d\mu(z)}{|1-\overline{\lambda}z|^{4}}<\infty,$$
or, equivalently,
$$\sup_{I\subset\TT}\ \mu(S(I))/|I|^{2} <\infty,$$
for any subarc  $I\subset \TT$ with arclengh $|I|$, and
  $S(I)$ is the Carleson box.

The measure $\mu$ is called vanishing (or compact) Carleson measure for $\cA_0^2$ if
the identity map $i_\alpha:\cA_0^2\to \textrm{L}^2(\mu)$ is a compact  operator. This happens if and only if
$$\lim_{|\lambda| \to 1}(1-|\lambda|^2)^{2}\int_\DD\frac{d\mu(z)}{|1-\overline{\lambda}z|^{4}}=0\iff\lim_{|I|\to0}\ \mu(S(I))/|I|^{2} =0.$$ 
Let $\varphi:\DD\to \DD$ be analytic and denote by $n_{\varphi}(z)$  the multiplicity of $\varphi$ at $z$.
By the change of variable formula \cite{S},
$$\|F_{\lambda,0}\|_{2,0}^p= (1-|\lambda|^{2})^{2}\int_\DD\frac{n_{\varphi}(z)\, dA(z)}{|1-\overline{\lambda}z|^{4}}.$$
Therefore, as a consequence of Theorem \ref{cg},  $C_\varphi$ is bounded in $\cD$ if and only if $n_{\varphi}(z)dA(z)$ is a Carleson measure for $\cA^2_0$ and  $C_\varphi$ is compact in $\cD$ if and only if $n_{\varphi}(z)dA(z)$ is a vanishing Carleson measure for $\cA^2_0$.
More explicitly, we have
$$\left\{\begin{array}{lll}
\displaystyle C_\varphi &\textrm{ is bounded in } \cD&\iff \displaystyle\sup_{I\subset\TT}\frac{1}{|I|^{2}}\int_{S(I)}n_{\varphi}(z)\, dA(z) <\infty;\\

\displaystyle C_\varphi &\textrm{ is compact  in } \cD&\iff \displaystyle\lim_{|I|\to0}\frac{1}{|I|^{2}}\int_{S(I)}n_{\varphi}(z)\, dA(z) =0.\\
\end{array}
\right.
$$

\section{Hilbert-Schmidt membership}

In the case of the Hardy  space $\hH $,  one can completely describe the membership of $C_\varphi$
in the Hilbert-Schmidt class in terms of the size of the level sets of the inducing map $\varphi$.
Indeed,  $C_\varphi$ is Hilbert-Schmidt in $\hH $ if and only if
$$\sum_{n \geq 0}\|\varphi^n\|^{2}_{\hH }=\int_\TT\frac{|d\zeta|}{1-|\varphi(\zeta)|^2}<\infty.$$
Given an arbitrary measurable function $f$ on $\TT$, consider the associated distribution function $m_f$ defined by
$$m_f(\lambda)=|\{\zeta\in \TT\text{ : } |f(\zeta)|>\lambda|\},\qquad \lambda>0.$$
It then follows that $C_\varphi$ is in the  Hilbert-Schmidt class of  $\hH $ if and only if
$$\int_\TT\frac{|d\zeta|}{1-|\varphi(\zeta)|^2} \,=\, \int_{1}^{\infty} m_{(1-|\varphi|^2)^{-1}}(\lambda)\, d\lambda
\asymp \int_{0}^{1}\frac{|E_\varphi(s)|}{(1-s)^2}\, ds<\infty.$$
It was shown by Gallardo--Gonz\'alez \cite[Theorem]{GG1} that there is a mapping $\varphi$ taking $\DD$ to itself
such that $C_\varphi$ is compact in $\hH $, and that the level set $E_\varphi(1)$ has Hausdorff measure equal to one. Recall that the Hausdorff dimension of $E$ 
$$d(E)=\inf\{\alpha \text{ : } \Lambda_\alpha(E)=0\}$$
where $\Lambda_\alpha(E)$ is the $\alpha$--Hausdorff measure of $E$ given by 
$$\Lambda_\alpha(E)=\lim_{\epsilon\to 0} \inf\Big\{\sum_{i}|\Delta_i|^\alpha\text{ : } E\subset\bigcup_i \Delta_i \text{ , } |\Delta_i|<\epsilon\Big\}.$$
Given  $E\subset \TT$ and $t>0$, let us write $E_t=\{\zeta\text{ : } d(\zeta,E)\leq t\}$ where $d$ denotes the arclength distance and $|E_t|$ denotes the Lebesgue measure of $E$. 

Let $E$ be a closed subset of $\TT$ with $|E_t|=O ((\log (e/t))^{-3})$ and  $E$ has Hausdorff  dimension one. (such examples can be given by generalized Cantor sets  \cite{C}). Let $\omega(t)=(\log (e/t))^{-2}$, and consider the outer function given by 
$$|f_{\omega,E}(\zeta)|=e^{-w(d(\zeta,E))},\qquad \text{ a.e on }\TT.$$
Since $\omega$ satisfies the  Dini condition 
$$\int_0 \frac{\omega(t)}{t}dt<\infty,$$
 it follows that $f_{\omega,E}\in A(\DD):= \Hol(\DD)\cap C(\overline{\DD})$,  disc algebra (see \cite{G} p.105--106) and so $E_{f_{\omega,E}}(1)=E$. On the other hand 
$$\int_\TT\frac{|d\zeta|}{1-|f_{\omega,E}(\zeta)|^2}\asymp  \int_\TT\frac{|d\zeta|}{\omega(d(\zeta,E))}\asymp \int_0{|E_t|}\frac{\omega'(t)}{\omega(t)^2}dt,$$
(see \cite[Proposition A.1 ]{EKR} for the last equality). Since   the last integral converges,  $C_\varphi$ is a  Hilbert-Schmidt operator in  $\hH $.

We have the following more precise result.

\begin{thm}\label{HardyHS}  Let $E$ be a closed subset of $\TT$ with Lebesgue measure zero.
There exists a mapping $\varphi: \DD \to \DD$, $\varphi\in A(\DD)$ such that $C_\varphi$ is a Hilbert-Schmidt operator on
$\hH $ and that $E_\varphi(1)=E$.
\end{thm}
 \begin{proof}
The proof is based a well known  construction of peak functions in the disc algebras. Let $\TT \setminus E= \displaystyle \cup _{n\geq 1} (e^{ia_n},e^{ib_n})$.
For $t\in (a_n,b_n)$, we define
$$g(e^{it})= \tau_{n}\frac{(b_n-a_n)^{1/2}}{((b_n-a_n)^2-(2t-(b_n+a_n))^2)^{1/4}},$$
where $(\tau_{n})_n \subset (0,\infty)$ will be chosen later, and $g(e^{it}):=+\infty$ if $e^{it}\in E$.\\
Note that
$$\displaystyle \int _0^{2\pi}g(e^{it})^2 \, dt = \pi\displaystyle \sum_{n =1}^\infty \tau_{n}^2 \, (b_n-a_n).$$
Since $\displaystyle \sum _{n=1}^\infty (b_n-a_n)=2\pi$, there exists a sequence $(\tau_{n})_n$ such that
$$\displaystyle \lim _{n\to +\infty}\tau_{n} = +\infty \ \ \mbox{and}\ \ \displaystyle \sum_{n=1}^\infty \tau_{n}^2 \, (b_n-a_n)<\infty.$$
Let $U$ denote the harmonic extension of $g$ on the unit disc given by 
$$
U(re^{i\theta})= \frac{1}{2\pi}\displaystyle \int _0^{2\pi} \frac{1-r^2}{|e^{it}-re^{i\theta}|^2} \, g(e^{it}) \, dt =
\displaystyle \sum _{n\in \ZZ}\widehat{g}(n) \, r^{|n|} \, e^{in\theta}.
$$
Since $\tau_{n} \to \infty$, one can easily verify that $\displaystyle \lim_{t\to \theta}g(e^{it})=+\infty$,
for $e^{i\theta}\in E$. Hence, $\displaystyle \lim _{r\to 1^-}U(re^{i\theta})=+\infty$, for $e^{i\theta}\in E$.\\
Let $V$ be the harmonic conjugate of $U$, with $V(0)=0$.
It is given by
$$
V(re^{i\theta}) = \displaystyle \sum _{n\neq 0} \frac{n}{|n|} \, \widehat{g}(n) \, r^{|n|} \, e^{in\theta}.
$$
Now, since $g$ is a $C^1$ function on $\TT \setminus E$,
we see that the holomorphic function $f = U+iV$ is continuous on ${\overline {\DD}} \setminus E$.
Knowing that $\displaystyle \lim _{r\to 1^-}U(re^{it})=+\infty$,
for $e^{it}\in E$, we get that $\displaystyle\varphi = \frac{f}{f+1} \in A(\DD)$,  disc algebra, and $E_\varphi (1)= E$.
Finally
\begin{eqnarray*}
\frac{1}{2\pi} \int_0^{2\pi}\frac{dt}{1-|\varphi(e^{it})|^2} & = &  \frac{1}{2\pi} \int_0^{2\pi}\frac{(U(e^{it})+1)^2+V^2(e^{it})}{(U(e^{it})+1)^2-U^2(e^{it})}dt\\
& \leq & \displaystyle\frac{1}{2\pi} \int_0^{2\pi}(U(e^{it})+1)^2+V^2(e^{it})dt\\
& \leq & 1+ 2 \sum _{n\in \ZZ}| \, \widehat{g}(n)|^2,
\end{eqnarray*}
which shows that $C_{\varphi}$ is Hilbert-Schmidt because $g\in L^2(\TT)$.
\end{proof}

Let $E$ be a closed subset of the unit circle $\T$.
Fix a non-negative function $w\in C^{1}(0,\pi]$ such that
$$\int_\T w(d(\zeta,E))\, |d\zeta|<\infty,$$
where d denotes the arclength distance.
Now, let $f_{w,E}$ be the outer function given by
\begin{equation}\label{exterieure}
|f^{*}_{w,E}(\zeta)|=e^{-w(d(\zeta,E))},\qquad \text{ a.e. on }\T.
\end{equation}
The following lemma gives an estimate for the Dirichlet integral of $f_{w,E}$ in terms of $w$
and the distance function on $E$.
The proof is based on Carleson's formula, and can be achieved by slightly modifying the arguments
used in  \cite[Theorem 4.1]{EKR1}.
\begin{lem}\label{lmnorme}  Assume that the function  $\omega$ is nondecreasing and
$\omega(t^\gamma)$ is  concave for all $\gamma>2$.
Then
$$\cD(f_{w,E})\asymp  \int_{\T} \omega{'}(d(\zeta,E))^2 \,e^{-2w(d(\zeta,E))}\, d(\zeta,E)\, |d\zeta|.$$
\end{lem}
Since the  sequence  $\{{z^n}/\sqrt{n+1}\}_{n=0}^\infty$ is an orthonormal basis  of $\cD$,
the operator $C_\varphi$ is Hilbert-Schmidt on the Dirichlet space if and only if
$$
\frac{1}{\pi}\int_\DD\frac{|\varphi'(z)|^2}{(1-|\varphi(z)|^2)^2}\, dA(z)=
\sum_{n\geq 1} \frac{\cD(\varphi^n)}{n} \,<\, \infty.
$$

\begin{thm}\label{thnorme}   Assume that the function  $\omega$ is nondecreasing and
$\omega(t^\gamma)$ is  concave for some $\gamma>2$.
Then  $C_{f_{w,E}}$ is in the Hilbert-Schmidt class in $\cD$ if and only if
$$\displaystyle \int_{\T} \frac{\omega{'}(d(\zeta,E))^2}{w(d(\zeta,E))^2}\, d(\zeta,E)\, |d\zeta| \,<\, \infty.$$
\end{thm}
\begin{proof} 
We first note that $f_{w,E}^n=f_{nw,E}$. Therefore, by Lemma \ref{lmnorme}, we have
\begin{eqnarray*}
\int_\DD\frac{|f_{w,E}'(z)|^2}{(1-|f_{w,E}(z)|^2)^2}\, dA(z)&=&\sum_{n=1}^\infty \frac{\cD(f_{nw,E})}{n}\\
& \asymp &  \int_{\T} \omega{'}(d(\zeta,E))^2\, d(\zeta,E)\, \sum_{n=1}^\infty n e^{-2nw(d(\zeta,E))}\, |d\zeta|\\
& \asymp &  \int_{\T} \frac{\omega{'}(d(\zeta,E))^2}{[1-e^{-2w(d(\zeta,E))}]^2} \, d(\zeta,E) \, |d\zeta|.
\end{eqnarray*}
Since $1-e^{-2w(d(\zeta,E))}\asymp w(d(\zeta,E))$, the result fllows.
\end{proof}

Given a (Borel) probability measure $\mu$ on $\TT$, we define its $\alpha$-energy, $0\leq \alpha < 1$, by
$$I_\alpha(\mu)=\sum_{n=1}^\infty \frac{|\widehat{\mu}(n)|^2}{n^{1-\alpha}}.$$
For a closed set $E \subset \TT$, its $\alpha$-capacity $\caap_\alpha (E)$ is defined by
$$\caap_\alpha (E):= 1/\inf\{I_\alpha(\mu) \text{ : } \mu \text{ is a probability measure on } E \}.$$
If $\alpha=0$, we simply note $\caap(E)$ and this means the logarithmic capacity of $E$.

The weak--type inequality for capacity \cite{C} states that, for $f\in \cD$ and $t\geq 4\|f\|^{2}_{\cD}$,
$$\caap(\{\zeta \text{ : } |f(\zeta)|\geq t\})\leq \frac{16\|f\|^{2}_{\cD}}{t^2}.$$
As a result of this inequality, we see that if $\liminf\|\varphi^n\|_\cD=0$, then  $\caap (E_\varphi(1))=0$.
Indeed, since   $E_\varphi(1)=E_{\varphi^n}(1)$,  the weak capacity inequality  implies that
$$\caap(E_\varphi(1))= \caap (E_{\varphi^n}(1))\leq {16}\|\varphi^n\|_\cD^2.$$
Now let  $n \to \infty$.
Hence, in particular, if the  operator $C_\varphi$ is in the
Hilbert-Schmidt class in $\cD$, then $\caap (E_\varphi(1))=0$.
This result was first obtained by Gallardo--Gonz\'alez \cite{GG, GG2} using a completely different method.
Theorems \ref{quan} and \ref{rec} give
quantitative versions of this result.
\begin{thm}\label{quan}
If $C_\varphi$ is a  Hilbert-Schmidt operator in $\cD$, then
\begin{equation}\label{if}
\int_{0}^{1}\frac{\caap (E_\varphi(s))}{1-s}\log \frac{1}{1-s} \,ds<\infty.
\end{equation}
\end{thm}

\begin{proof}
Fix  $\lambda\in \TT$ and let
$$\varphi_\lambda(\zeta)=\log \,\Re \, \frac{1+\lambda\varphi(\zeta)}{1-\lambda\varphi(\zeta)},\qquad \zeta\in \T.$$
Since $$\int_\D\frac{|\varphi'(z)|^2}{(1-|\varphi(z)^2|)^2}\, dA(z)<\infty,$$
it follows that $\varphi_\lambda\in \cD(\TT)$, see \cite{GG}, where
$$\cD(\T):=\{f\in L^2(\T)\text{ : }\|f\|^{2}_{\cD(\TT)}=\sum_{n\in\ZZ}|\widehat{f}(n)|^2 (1+|n|)<\infty\}.$$
Setting   $\Delta_\lambda:=\{\zeta\in \TT\text{ : } |1-\lambda\varphi(\zeta)|\geq 1\}$, we see that
$$|\varphi_\lambda(\zeta)|\asymp \log\frac{1}{1-|\varphi(\zeta)|^2},\qquad \forall\zeta\in \Delta_\lambda.$$
Applying the strong capacity inequality \cite[Theorem 2.2]{W} to $\varphi_\lambda$, we get
\begin{eqnarray*}
\infty>\|\varphi_\lambda\|^{2}_{\cD(\T)}&\geq &c\int^{\infty}\text{ cap\;}\{\zeta\in \TT\text{ : } |\varphi_\lambda(\zeta)|>s\}\, ds^2\\
&=&c
\int^{\infty}\text{ cap\;}\Big\{\zeta\in\T\text{ : } \Big|\log\frac{1-|\varphi(\zeta)|^2}{|1-\lambda\varphi(\zeta)|^2}\Big|>s\Big\} \, ds^2\\
&\geq&c \int^{\infty}\text{ cap\;}\Big\{\zeta\in\T\cap\Delta_\lambda\text{ : }\Big|\log\frac{1-|\varphi(\zeta)|^2}{|1-\lambda\varphi(\zeta)|^2}\Big|>s\Big\} \, ds^2\\
&\geq&c \int^{\infty}\text{ cap\;}\Big\{\zeta\in\T\cap\Delta_\lambda\text{ : } \log\frac{1}{1-|\varphi(\zeta)|^2}>4s\Big\}\, ds^2\\
&\geq &c_1\int^{1}\text{ cap\;}\Big\{\zeta\in \T\cap\Delta_\lambda\text{ : }  |\varphi(\zeta)|> u\Big\}\, d\big(\log \frac{1}{1-u}\big)^2.
\end{eqnarray*}
Since  $\TT=\Delta_1\cup\Delta_{-1}$, the subadditivity of the capacity implies that
\begin{multline*}
\infty>\|\varphi_1\|^{2}_{\cD(\T)}+ \|\varphi_{-1}\|^{2}_{\cD(\T)}\geq \\c_2\int^{1}\text{ cap\;}\Big\{\zeta\in\TT\text{ : }  |\varphi(\zeta)|> u\Big\}\, d\big(\log \frac{1}{1-u}\big)^2,
\end{multline*}
 and hence the theorem follows.
\end{proof}

\begin{rem} 
\end{rem}

 Since $\{z^n/(1+n)^{\frac{1-\alpha}{2}}\}_{n=0}^{\infty}$ is an orthonormal basis  in $\cD_\alpha$, $\alpha\in (0,1)$, $C_\varphi$ is a  Hilbert-Schmidt operator in $\cD_\alpha$ if and only if
$$\sum_ {n=1}^\infty \frac{\cD_\alpha(\varphi^n)}{n^{1-\alpha}}\asymp \int_\DD\frac{|\varphi'(z)|^2}{(1-|\varphi(z)|^2)^{2+\alpha}} \, dA_\alpha(z)<\infty.$$
Therefore,   for fixed  $\lambda\in \TT$, the  function
$$\varphi_\lambda(\zeta)= \Big(\Re \frac{1+\lambda\varphi(\zeta)}{1-\lambda\varphi(\zeta)}\Big)^{-\alpha/2},\qquad (\zeta\in \T),$$
belongs to the weighted harmonic Dirichlet space
$$\cD_\alpha(\TT):=\{f\in L^2(\T)\text{ : }\|f\|^{2}_{\cD_\alpha(\TT)}=\sum_{n\in\ZZ}|\widehat{f}(n)|^2 (1+|n|)^{1-\alpha}<\infty\},$$
 (see \cite{GG2}) . Applying again the strong capacity inequality \cite[Theorem 2.2]{W} for $\cD_\alpha$  to $\varphi_\lambda$, we get as before
$$
\int_{0}^{1}\frac{\caap_\alpha (E_\varphi(s))}{(1-s)^{1+\alpha}}\,ds<\infty.
$$

The following theorem   is the analogue of  Proposition \ref{HardyHS} for the Dirichlet space.  
It shows that  condition \eqref{if} is optimal.
\begin{thm}\label{rec}
Let $h : [1,+\infty[\to [1,+\infty[$ be a function such that $\displaystyle \lim _{x\to \infty}h(x)=+\infty$. 
Let $E$ be a closed subset of $\TT$ such that $\caap  (E)= 0$. 
Then there is $\varphi\in A(\DD)\cap \cD$,  $\varphi( \DD )\subset\DD$ such that :
\begin{enumerate}
\item $E_{\varphi}(1)= E$;
\item   $C_{\varphi}$ is in the Hilbert--Schmidt class in $\cD$;
\item $
\displaystyle \int ^1\frac{\caap(E_{\varphi}(s))}{1-s} \log\frac{1}{1-s}h\Big(\frac{1}{1-s}\Big)\, ds =+\infty.
$
\end{enumerate}
\end{thm}
\begin{proof}
Let $k(x) = h(e^x)$, there exists a  continuous decreasing function $\psi$ such that
$$
\displaystyle \int ^{+\infty}\psi(x)\, dx^2<\infty\ \ \mbox{ and}\ \ \displaystyle \int ^{+\infty}\psi(x)\,k(x)\,dx^2=\infty.
$$
Set $\eta (t)= \psi ^{-1}(\caap  (E_t))$. We have
\begin{eqnarray*}
 \int _0 \caap (E_t)\, |d\eta ^2(t)|&\asymp & \int _0 \psi (\eta (t))\, |d\eta ^2(t)|\\
 &\asymp &\int ^{+\infty} \psi (x)\, dx^2<\infty,
\end{eqnarray*}
and,
\begin{eqnarray*}
 \int _0 \caap (E_t)\, h(e^{\eta (t)}))\, |d\eta ^2(t)|&\asymp& \int _0 \psi (\eta (t))\, k(\eta (t))\, |d\eta ^2(t)|\\
 &\asymp&\int ^{+\infty} \psi(x)\, k(x)\, dx^2=\infty.\\
\end{eqnarray*}
Since
$$
\displaystyle \int _0 \caap (E_t)\, |d\eta ^2(t)|<\infty,
$$
by \cite[Theorem 5]{EKR}, there exists a function $f\in {\cal D}$ such that
$$
 \Re f (\zeta) \geq \eta (d(\zeta , E)) \quad \text{ and } \quad |\Im f (\zeta)|<\pi /4, \qquad  \text{q.e. on \;\;}  \TT.
$$
By  harmonicity,
$$ |\Im f (z)|<\pi /4\,,\qquad   |z|<1,$$
 Now take
$$\varphi = \exp (-e^{-f}).$$
By a simple modification in the construction of $f$ as in \cite{BC}, we can suppose that  $\varphi\in A(\DD)$. Hence  $E_{\varphi}(1)= E$ and
\begin{eqnarray*}
\int_\DD \frac{|\varphi'(z)|^2}{(1-|\varphi(z)|^2)^2}
& \asymp &\int_\DD \frac{|f'(z)|^2 \, e^{-2\,\Re f(z)} \,
e^{-2\,e^{\,-\Re f(z)}\,\cos (\Im f(z))}}{e^{-2 \, \Re f(z)} \, \cos^2 ( \Im f(z)) }\, dA(z)\\
& \leq &\int_\DD {|f'(z)|^2 \, \exp(-\sqrt{2}\, e^{- \Re f(z)})}\, dA(z)\\
& \leq &  c\int_\DD |f'(z)|^2\, dA(z) \,<\, \infty.
\end{eqnarray*}
Hence  $C_{\varphi}$ is in the Hilbert--Schmidt class.
Finally, since
$$
E_{\varphi}(s)\supseteq \{\zeta\in \TT\text{ : }\eta (d(\zeta,E))\geq \log (1/1-s)\},
$$
we get
$$
 \int_0 \caap (E_{\varphi}(s))h(1/1-s)\, d\big(\log (1/1-s)\big)^2 \geq
  \int_0 \caap (E_t)h(e^{\eta (t)})\, |d\eta ^2(t)|=+\infty.
$$
\end{proof}

\end{document}